\RequirePackage{fix-cm}
\documentclass[smallextended]{svjour3}       
\smartqed  
\usepackage{amssymb}
\usepackage[leqno]{amsmath}
\usepackage{cases}
\usepackage{graphicx}
\usepackage{subfigure}
\usepackage{epstopdf}
\usepackage{epsfig}
\usepackage{float}

\usepackage{mathrsfs}
\usepackage{mathabx}

\usepackage{tabularx}

\usepackage{color}
\usepackage{hyperref}

\newtheorem{assumption}{Assumption}

\begin{document}

\title{Uniform convergence of finite element method on Bakhvalov-type mesh for a 2-D singularly perturbed convection-diffusion problem with exponential layers
\thanks{The current research  was partly supported by NSFC (11771257), Shandong Provincial NSF (ZR2021MA004).}
}

\titlerunning{Finite element method on Bakhvalov-type mesh}        

\author{Jin Zhang         \and
        Chunxiao Zhang 
}


\institute{Corresponding author:  Jin Zhang \at
              School of Mathematics and Statistics, Shandong Normal University,
Jinan 250014, China\\
              \email{jinzhangalex@sdnu.edu.cn}           
           \and
           Chunxiao Zhang\at
           School of Mathematics and Statistics, Shandong Normal University,
Jinan 250014, China\\
            \email{chunxiaozhangang@outlook.com}  
}

\date{Received: date / Accepted: date}

\maketitle

\begin{abstract}
On Bakhvalov-type mesh, uniform convergence analysis of finite element method for a 2-D singularly perturbed convection-diffusion problem with exponential layers is still an open problem. Previous attempts have been unsuccessful. The primary challenges are the width of the mesh subdomain in the layer adjacent to the transition point, the restriction of the Dirichlet boundary condition, and the structure of exponential layers. To address these challenges, a novel analysis technique is introduced for the first time, which takes full advantage of the characteristics of interpolation and the connection between the smooth function and the layer function on the boundary. Utilizing this technique in conjunction with a new interpolation featuring a simple structure, uniform convergence of optimal order $k+1$ under an energy norm can be proven for finite element method of any order $k$. Numerical experiments confirm our theoretical results.
\end{abstract}
%
%
%
\section{Introduction}	
Take the following elliptic boundary value problem into consideration:
\begin{equation}\label{model problem}
	\begin{aligned}
		-\varepsilon\Delta u-\boldsymbol{b}\cdot\nabla u+cu=&f\quad&&\text{in}\quad\Omega=(0,1)^2,\\
		u=&0\quad&&\text{on}\quad\partial\Omega,
	\end{aligned}
\end{equation}
where $0<\varepsilon\ll1$ is a small positive constant, and $\boldsymbol{b}$ denotes $(b_1(x,y),b_2(x,y))$. Let $b_1$, $b_2$, and $f$ be sufficient smooth functions satisfying
\begin{equation}\label{ensure coercivity}
	b_1(x,y)\ge\beta_1>0,\ b_2(x,y)\ge\beta_2>0,\ c(x,y)+\frac{1}{2}\nabla\cdot\boldsymbol{b}(x,y)\ge\gamma>0\ \text{on}\ \Omega,
\end{equation}
with some positive constants $\beta_1$, $\beta_2$ and $\gamma$. Under the given conditions, equation \eqref{model problem} possesses a unique solution in the space $H_0^1(\Omega)\cap H^2(\Omega)$ for all $f\in L^2(\Omega)$ \cite{Roo1Sty2:2008-Robust}. This solution exhibits exponential boundary layers of width $\mathcal{O}(\varepsilon\ln(1/\varepsilon))$ at $x=0$ and $y=0$ as well as a corner layer at $(0,0)$. Since the diffusion parameter $\varepsilon$ can be arbitrarily small, this problem is characterized as singularly perturbed and is convection-dominated.  

Studying such singularly perturbed convection-diffusion problems has significant theoretical and practical value. In order to obtain the required accuracy of the numerical solution, various uniformly convergent numerical methods with respect to the small parameter $\varepsilon$ were proposed (see \cite{Roo1Sty2:2008-Robust}). Among them,  the combination of finite element method  and layer-adapted meshes is particularly effective. Shishkin-type mesh and Bakhvalov-type mesh are two popular types of layer-adapted meshes \cite{Lin:2009-Layer}. Shishkin-type mesh has a simple structure and has been widely used to study convergence (see \cite{Li1Nav2:1998-Uniformly,Lin1:2000-Uniform,Zhang1:2003-Finite,Roo1Sch2:2015-Convergence,Lin1Sty2:2012-Balanced,Zhang1Liu2:2016-Optimal,Liu1Zhang2:2018-Uniform} and their references therein). In contrast to Shishkin-type mesh, Bakhvalov-type mesh has a complex structure as it is graded in the mesh subdomain used to resolve layers. Nevertheless, its transition point is independent of the mesh parameter $N$, which is an essential property in certain cases. Furthermore, Bakhvalov-type mesh performs numerically better than Shishkin-type mesh, particularly when higher-order finite element schemes are employed.

Up to this point, uniform convergence analysis for singularly perturbed problems on Bakhvalov-type mesh is still an open problem, which is attributed to the intricate structure of the mesh (see \cite[Question 4.1]{Roo1Sty2:2015-Some} for further information). In an effort to deal with the problems raised in \cite[Question 4.1]{Roo1Sty2:2015-Some}, several convergence studies utilizing finite element method have been carried out. In the case of 1D, Roos \cite{Roo1:2006-Error} introduced a quasi-interpolation for linear finite element method and proved uniform convergence of optimal order, while this approach cannot be generalized to higher-order finite element schemes or higher-dimensional cases. Subsequently, Zhang \cite{zhang1Liu2:2020-Optimal} developed a novel interpolation with a simplified structure suitable for finite element method of any order. In the case of 2D, Zhang \cite{zhang1Liu2:2023-Convergence} established uniform convergence for a singularly perturbed convection-diffusion problem with parabolic layers by extending the novel interpolation in \cite{zhang1Liu2:2020-Optimal} to the case of 2D. However, when considering a 2-D singularly perturbed convection-diffusion problem with exponential layers, previous research, including \cite{Roo1Sch2:2012-Analysis} and \cite{zhang1Liu2:2023-Convergence}, both failed to achieve uniform convergence. The main difficulties come from the special width of the Bakhvalov-type mesh subdomain in the layer adjacent to the transition point, the restriction of the Dirichlet boundary condition, and the structure of exponential layers. These difficulties result in suboptimal results when using the standard error analysis methods to estimate some convection terms of the layer part, and this situation does not arise in parabolic layers.

This article aims to present a novel analysis technique for solving the 2-D singularly perturbed convection-diffusion problem with exponential layers. To be specific, for the smooth part in particular regions, we convert the corresponding estimates in the two-dimensional case into several types of one-dimensional interpolation error estimates. Furthermore, we make full use of the relationship between the smooth function and the layer function on the boundary. For the layer part, we apply a new interpolation with a simplified structure that were proposed in \cite{zhang1Liu2:2023-Convergence} and adopt the standard error analysis approach. Our scheme is demonstrated to uniformly converge to an optimal order of $k+1$. To the best of our knowledge, this is the first proof of uniform convergence for a typical problem \eqref{model problem} on Bakhvalov-type mesh in the two-dimensional setting. It is worth mentioning that the novel analysis technique proposed in this paper provides a new path for the future convergence analysis.

The structure of this article is as follows: In Section \ref{sec 2}, we present the regularity of the solution to \eqref{model problem}, define a Bakhvalov-type mesh, and introduce the $k$-th order finite element method. New interpolation and preliminary results are displayed in Section \ref{sec 3}. Detailed derivations of uniform convergence under an energy norm can be found in Section \ref{sec 4}. Section \ref{sec 5} carries out the numerical experiments that illustrate our theoretical results.

Assume that $D$ is any measurable subset of $\Omega$. $(\cdot,\cdot)_D$, $\Vert\cdot\Vert_{L^1(D)}$, $\Vert\cdot\Vert_{\infty,D}$, $\Vert\cdot\Vert_{D}$ and $\vert\cdot\vert_{1,D}$ represent the standard inner product in $L^2(D)$, the standard norms in $L^1(D)$, $L^\infty(D)$, $L^2(D)$ and the standard seminorms in $H^1(D)$, respectively. When $D=\Omega$, drop $D$ from these notations for the sake of clarity. In addition, all constants used in this paper, including the generic constant $C$ and the fixed constant $C_i$, are all positive and unaffected
by the singular perturbation $\varepsilon$ as well as the mesh parameter N.

\section{Regularity, Bakhvalov-type mesh and finite element method}\label{sec 2}
\subsection{Regularity of the solution}
In the process of analysis, we assume that $k\ge1$ is a fixed integer.
\begin{assumption}\label{bound}
	We can decompose the solution to \eqref{model problem} into the following parts:
	\begin{equation}\label{decomposition of u}
		u=S+E_{1}+E_{2}+E_{12}\quad \forall(x,y)\in\overline{\Omega},
	\end{equation}
	where $S$ is the smooth part, $E_{1}$ is the exponential layer part at $x=0$, $E_{2}$ is another exponential layer part at $y=0$, and $E_{12}$ is the corner layer part.
	
	For any $(x,y)\in \bar{\Omega}$, the following bounds hold:
	\begin{subequations}
		\begin{align}
			\left|\frac{\partial^{m+n}S}{\partial x^{m}\partial y^{n}}(x,y) \right| &\le C,\\
			\left|\frac{\partial^{m+n}E_{1}}{\partial x^{m}\partial y^{n}}(x,y) \right| &\le C\varepsilon^{-m}e^{\frac{-\beta_1 x}{\varepsilon}},\\
			\left|\frac{\partial^{m+n}E_{2}}{\partial x^{m}\partial y^{n}}(x,y) \right| &\le C\varepsilon^{-n}e^{\frac{-\beta_2y}{\varepsilon}},\\
			\left|\frac{\partial^{m+n}E_{12}}{\partial x^{m}\partial y^{n}}(x,y) \right| &\le C\varepsilon^{-(m+n)}e^{\frac{-\beta_1 x-\beta_2 y}{\varepsilon}},
		\end{align}
	\end{subequations}
	where $m$, $n$ are nonnegative integers with $0\le m+n\le k+1$.
\end{assumption}
We can refer to \cite{Lin1Sty2:2001-Asymptotic} and \cite[Sec. 7]{Sty:2005-Steady} for concrete information regarding Assumption \ref{bound}.

\subsection{Bakhvalov-type mesh}
Bakhvalov mesh was first proposed in \cite{Bakh:1969-Optimization} to improve the convergence order. Bakhvalov-type mesh, as an approximation of Bakhvalov mesh, has been widely used to avoid solving nonlinear equations appearing on Bakhvalov mesh \cite{Lin:2009-Layer}.

In this study, we consider a Bakhvalov-type mesh introduced in \cite{Lin:2009-Layer}, which is defined by
\begin{equation}\label{Bakhvalov-type mesh}
	\begin{aligned}
		&x_i=\left\{\begin{aligned}
			&-\frac{\sigma\varepsilon}{\beta_1}\ln(1-2(1-\varepsilon)i/N)\quad&&\text{for} \quad i=0,1,\dots,N/2,\\
			&1-2(1-x_{N/2})(N-i)/N\quad&&\text{for}\quad i=N/2+1,\dots,N,
		\end{aligned}
		\right.\\
		&y_j=\left\{\begin{aligned}
			&-\frac{\sigma\varepsilon}{\beta_2}\ln(1-2(1-\varepsilon)j/N)\quad&&\text{for}\quad j=0,1,\dots,N/2,\\
			&1-2(1-y_{N/2})(N-j)/N\quad&&\text{for}\quad j=N/2+1,\dots,N,
		\end{aligned}
		\right.
	\end{aligned}
\end{equation}
where $\sigma\ge k+1$ and $N$ is an even positive integer. Transition points are $x_{N/2}=-\frac{\sigma\varepsilon}{\beta_1}\ln\varepsilon$ and $y_{N/2}=-\frac{\sigma\varepsilon}{\beta_2}\ln\varepsilon$, indicating a shift in mesh from coarse to fine. Also, one has $e^{-\beta_1x_{N/2}/\varepsilon}=\varepsilon^\sigma$ and $e^{-\beta_1x_{N/2-1}/\varepsilon}=N^{-\sigma}$.

By connecting the mesh points $\left\lbrace(x_i,y_j)|i,j=0,1,\dots,N\right\rbrace$ with lines parallel to the x-axis and the y-axis, we obtain a rectangulation of $\Omega$ denoted as $\mathcal{T}_N$ ( see Figure \ref{fig:paper3-rectangulation}). We define $K_{i,j}=[x_{i},x_{i+1}]\times[y_{j},y_{j+1}]$ as a specific mesh rectangle, where $i,j=0,1,\dots,N-1$, and $K\in\mathcal{T}_N$ as a general mesh rectangle. The lengths of $K_{i,j}$ in the $x$- and $y$-directions are represented by $h_{x,i}:=x_{i+1}-x_i$ and $h_{y,j}:=y_{j+1}-y_j$, respectively.
\begin{figure}
	\centering
	\includegraphics[width=0.7\linewidth]{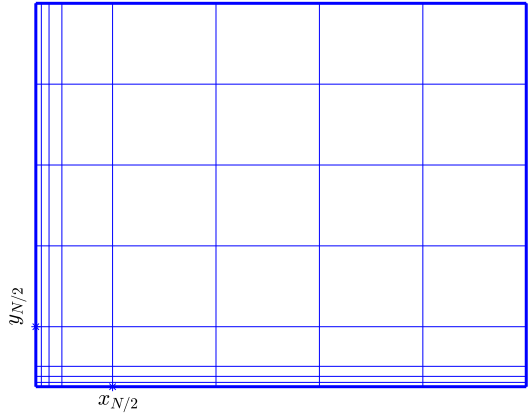}
	\caption{Rectangulation of $\Omega$}
	\label{fig:paper3-rectangulation}
\end{figure}

\begin{assumption}\label{restriction}
	In this paper, suppose that $\varepsilon\le N^{-1}$, because it is not a practical restriction.
\end{assumption}

In the following lemma, we will present some important properties of Bakhvalov-type mesh \eqref{Bakhvalov-type mesh} (refer to \cite[Lemma 3]{zhang1Liu2:2020-Optimal}).
\begin{lemma}\label{step}
	Let Assumption \ref{restriction} hold, then for Bakhvalov-type mesh \eqref{Bakhvalov-type mesh}, one can obtain
	\begin{align}
		&C_1\varepsilon N^{-1}\le h_{x,0}\le C_2\varepsilon N^{-1},\\
		&h_{x,0}\le h_{x,1}\le\dots\le h_{x,N/2-2},\\
		&\frac{1}{4}\sigma\varepsilon\le h_{x,N/2-2}\le \sigma\varepsilon,\\
		&\frac{1}{2}\sigma\varepsilon\le h_{x,N/2-1}\le 2\sigma N^{-1},\\
		&N^{-1}\le h_{x,i}\le 2N^{-1}\quad N/2\le i\le N-1,\\
		&C_3\sigma\varepsilon\ln N\le x_{N/2-1}\le C_4\sigma\varepsilon\ln N,\quad x_{N/2}\ge C\sigma\varepsilon|\ln\varepsilon|,\\
		&h_{x,i}^{\rho}e^{-\beta_1x/\varepsilon}\le C\varepsilon^{\rho}N^{-\rho}\quad 0\le i\le N/2-2\quad\text{and}\quad 0\le\rho\le\sigma.
	\end{align}
Bounds for $h_{y,j}$ $0\le j\le N-1$, are analogous.
\end{lemma}

\subsection{Finite element method}
The variational form of \eqref{model problem} can be described as
\begin{equation}\label{variational form 1}
	\left\{
	\begin{aligned}
		&\text{find}\ u\in H_0^1(\Omega) \ \text{such that for all}\ v\in H_0^1(\Omega)\\
		&a(u,v)=(f,v),
	\end{aligned}
	\right.
\end{equation}
where 
\begin{equation}
	a(u,v):=\varepsilon(\nabla u,\nabla v)+(-\boldsymbol{b}\cdot\nabla u,v)+(cu,v).
\end{equation}

Define a finite element space $V^N\subset H_0^1(\Omega)$ on $\mathcal{T}_N$:
\begin{equation*}
	V^N:=\left\lbrace v^N\in C(\overline{\Omega}):\ v^N|_{\partial\Omega}=0\ \text{and}\ v^N|_K\in\mathcal{Q}_k(K)\quad \forall K\in\mathcal{T}_N\right\rbrace,
\end{equation*}
with $\mathcal{Q}_k(K)=\text{span}\left\lbrace x^iy^j:\ 0\le i,j\le k\right\rbrace$. 

Then the $k$-th order finite element method is
\begin{equation}\label{variational form 2}
	\left\{\begin{aligned}
		&\text{find}\ u^N\in V^N \ \text{such that for all}\ v\in V^N\\
		&a(u^N,v^N)=(f,v^N),
	\end{aligned}
	\right.
\end{equation}
where 
\begin{equation}\label{variation}
	a(u^N,v^N):=\varepsilon(\nabla u^N,\nabla v^N)+(-\boldsymbol{b}\cdot\nabla u^N,v^N)+(cu^N,v^N).
\end{equation}

Condition \eqref{ensure coercivity} gives the coercivity
\begin{equation}\label{coercivity}
	a(v^N,v^N)\ge C\Vert v^N\Vert_\varepsilon^2\quad\forall v^N\in V^N,
\end{equation}
here the energy norm $\Vert\cdot\Vert_\varepsilon$ is defined as
\begin{equation}
	\Vert v\Vert_\varepsilon^2:=\left(\varepsilon|v|_1^2+\Vert v\Vert^2\right)^{1/2}.
\end{equation}

The following Galerkin orthogonality property can be obtained by \eqref{variational form 1} and \eqref{variational form 2}:
\begin{equation}\label{orthogonality}
	a(u-u^N,v^N)=0\quad\forall v^N\in V^N.
\end{equation}

Furthermore, by employing the Lax-Milgram lemma, existence and uniqueness of the finite element solution $u^N$ can be well demonstrated \cite{Bre1Sco2:2008-Mathematical}.
\section{New interpolation, interpolation errors and some preliminary results}\label{sec 3}
In this section, we will introduce a novel interpolation that were proposed in \cite{zhang1Liu2:2023-Convergence}. Interpolation errors and preliminary results are also presented. To begin with, some notations need to be defined: For $k\ge1$, set $x_i^s:=x_i+(s/k)h_{x,i}$ and $y_j^t:=y_j+(t/k)h_{y,j}$, where $i,j=0,1,\dots,N-1$ and $s,t=0,1,\dots,k-1$. Besides, set $x_N^0=x_N$ and $y_N^0=y_N$ to ensure notational uniformity. For any $v\in C^0(\overline{\Omega})$, its standard Lagrange interpolation $v^I\in V^N$ can be expressed as
\begin{equation}\label{standard Lagrange interpolation}
	\begin{aligned}
		v^I(x,y)=&\sum_{i=0}^{N-1}\sum_{s=0}^{k-1}\left(\sum_{j=0}^{N-1}\sum_{t=0}^{k-1}v(x_i^s,y_j^t)\theta_{i,j}^{s,t}(x,y)+v(x_i^s,y_N^0)\theta_{i,N}^{s,0}(x,y)\right)\\
		&+\sum_{j=0}^{N-1}\sum_{t=0}^{k-1}v(x_N^0,y_j^t)\theta_{N,j}^{0,t}(x,y)+v(x_N^0,y_N^0)\theta_{N,N}^{0,0}(x,y),
	\end{aligned}
\end{equation}
here $\theta_{i,j}^{s,t}(x,y)\in V^N$ is the piecewise tensor-product interpolation basis function of degree $k$ connected to the point $(x_i^s,y_j^t)$. Write $\theta_{i,j}^{s,t}(x,y)$ as $\theta_{i,j}^{s,t}$ for simplicity.

Recall \eqref{decomposition of u}, we can decompose the new interpolation $\Pi u$ as
\begin{equation}\label{decomposition of new interpolation u}
	\Pi u=S^I+\pi_1E_1+\pi_2E_2+\pi_{12}E_{12},
\end{equation}
where $S^I$ represents the standard Lagrange interpolation of $S$, and
\begin{align}
	&\pi_1E_1=E_1^I-\sum_{s=0}^{k-1}\left(\sum_{t=1}^{k-1}E_1(x_{N/2-1}^s,y_0^t)\theta_{N/2-1,0}^{s,t}-\sum_{j=1}^{N-1}\sum_{t=0}^{k-1}E_1(x_{N/2-1}^s,y_j^t)\theta_{N/2-1,j}^{s,t}\right),\label{new interpolation-1}\\
	&\pi_2E_2=E_2^I-\sum_{t=0}^{k-1}\left(\sum_{s=1}^{k-1}E_2(x_0^s,y_{N/2-1}^t)\theta_{0,N/2-1}^{s,t}+\sum_{i=1}^{N-1}\sum_{s=0}^{k-1}E_2(x_i^s,y_{N/2-1}^t)\theta_{i,N/2-1}^{s,t}\right),\label{new interpolation-2}\\
	&\pi_{12}E_{12}=E_{12}^I-\sum_{s=0}^{k-1}\sum_{t=0}^{k-1}E_{12}(x_{N/2-1}^s,y_{N/2-1}^t)\theta_{N/2-1,N/2-1}^{s,t}\label{new interpolation-3}.
\end{align}
\begin{remark}
	The concept of $\pi_1E_1$ involves subtracting some intractable terms that fail to achieve uniform convergence from the standard Lagrange interpolation $E_1^I$, and at the same time, ensuring $\Pi u$ satisfies the homogeneous Dirichlet boundary condition.  $\pi_2E_2$ and $\pi_{12}E_{12}$ are similarly constructed. 
\end{remark}

According to \cite[Theorem 2.7]{Apel:1999-Anisotropic}, we have the following interpolation errors.
\begin{lemma}\label{interpolation errors}
	Let $p\in(1,\infty]$ and $K_{i,j}\in\mathcal{T}_N$ with $i,j=0,1,\cdots,N-1$. For any $w\in W^{k+1,p}(\Omega)$, its standard Lagrange interpolation $w^{I}$ at the vertices of $K_{i,j}$ satisfies
	\begin{equation*}
		\begin{aligned}
			&\Vert
			w-w^{I}\Vert_{L^p(K_{i,j})}\le C\sum_{m+n=k+1}h_{x,i}^{m}h_{y,j}^{n}\Vert\frac{\partial^{k+1} w}{\partial x^{m}\partial y^{n}} \Vert_{L^p(K_{i,j})},\\
			&\Vert(w-w^{I})_{x}\Vert_{L^p(K_{i,j})}\le C\sum_{m+n=k}h_{x,i}^{m}h_{y,j}^{n}\Vert\frac{\partial^{k+1} w}{\partial x^{m+1}\partial y^{n}} \Vert_{L^p(K_{i,j})},\\
			&\Vert
			(w-w^{I})_{y}\Vert_{L^p(K_{i,j})}\le C\sum_{m+n=k}h_{x,i}^{m}h_{y,j}^{n}\Vert\frac{\partial^{k+1} w}{\partial x^{m}\partial y^{n+1}} \Vert_{L^p(K_{i,j})},
		\end{aligned}
	\end{equation*}
	here $m$ and $n$ denote non-negative integers.  
\end{lemma}

Next, some preliminary results will be given.
\begin{lemma}\label{error estimates}
	Let Assumptions \ref{bound} and \ref{restriction} hold. Let $\sigma\ge k+1$. Denote $E_i^I$ as the standard Lagrange interpolations of $E_{i}\,(i=1,2,12)$, respectively. Then, the following interpolation error estimates can be obtained on Bakhvalov-type mesh \eqref{Bakhvalov-type mesh}.
	\begin{align}
		&\Vert E_1-E_1^I\Vert+\Vert E_2-E_2^I\Vert\le CN^{-(k+1)},\\
		&\Vert E_{12}-E_{12}^I\Vert\le C\varepsilon N^{-k}+C\varepsilon^{1/2}N^{-(k+1)}+CN^{-(1+2\sigma)},\\
		&\Vert E_1-E_1^I\Vert_\varepsilon+\Vert E_2-E_2^I\Vert_\varepsilon+\Vert E_{12}-E_{12}^I\Vert_\varepsilon\le CN^{-k}.
	\end{align}
\end{lemma}
\begin{proof}
	Detailed demonstrations can refer to \cite[Lemmas 3.2 and 3.3]{zhang1Liu2:2023-Convergence}.
\end{proof}

We also require the following inequalities for the subsequent analysis: $\forall K\in\mathcal{T}_N$, for $i=1,2,12$, we have
\begin{align}
	&\Vert E_i^I\Vert_{\infty,K}\le \Vert E_i\Vert_{\infty,K},\label{max-1}\\
	&\Vert \pi_iE_i\Vert_{\infty,K}\le \Vert E_i\Vert_{\infty,K}\label{max-2},
\end{align}
where $E_i^I$ denote the standard Lagrange interpolation and $\pi_iE_i$ are defined in \eqref{new interpolation-1}-\eqref{new interpolation-3}.

\section{Uniform convergence}\label{sec 4}
In this section, we will prove uniform convergence of optimal order for problem \eqref{model problem} on Bakhvalov-type mesh \eqref{Bakhvalov-type mesh}. 

For the sake of briefness, set
\begin{equation*}
	\begin{aligned}
		[x_{N/2-1},x_{N/2}]\times[y_{N-1},y_N]&=:K_a,\\
		[x_{N-1},x_N]\times[y_{N/2-1},y_{N/2}]&=:K_b,\\
		\Omega\textbackslash(K_a\cup K_b)&=:\Omega_0.
	\end{aligned}
\end{equation*}  

Let $v^N:=\Pi u-u^N$. Recall \eqref{variation}, \eqref{coercivity}, \eqref{orthogonality}, \eqref{decomposition of u} and \eqref{decomposition of new interpolation u}, one can obtain
\begin{equation}
	\begin{aligned}
		&C\Vert v^N\Vert_\varepsilon^2\le a(v^N,v^N)=a(\Pi u-u,v^N)\\
		=&\varepsilon\int_{\Omega}\nabla(\Pi u-u)\nabla v^N\mathrm{d}x\mathrm{d}y+\int_{\Omega}c(\Pi u-u)v^N\mathrm{d}x\mathrm{d}y\\
		&-\int_{K_a}\boldsymbol{b}\cdot\nabla(\Pi u-u) v^N\mathrm{d}x\mathrm{d}y-\int_{K_b}\boldsymbol{b}\cdot\nabla (\Pi u-u)v^N\mathrm{d}x\mathrm{d}y\\
		&-\sum_{i=1,2,12}\int_{\Omega_0}\boldsymbol{b}\cdot\nabla(\pi_iE_i-E_i) v^N\mathrm{d}x\mathrm{d}y-\int_{\Omega_0}\boldsymbol{b}\cdot\nabla(S^I-S) v^N\mathrm{d}x\mathrm{d}y\\
		=:&\uppercase\expandafter{\romannumeral1}+\uppercase\expandafter{\romannumeral2}+\uppercase\expandafter{\romannumeral3}+\uppercase\expandafter{\romannumeral4}+\uppercase\expandafter{\romannumeral5}+\uppercase\expandafter{\romannumeral6}.
	\end{aligned}
\end{equation}

Estimates for $\uppercase\expandafter{\romannumeral1}-\uppercase\expandafter{\romannumeral6}$ are presented in the following lemmas.

\begin{lemma}\label{main-1}
	Let Assumptions \ref{bound} and \ref{restriction} hold. Let $\sigma\ge k+1$. Then one has
	\begin{equation}
		|\uppercase\expandafter{\romannumeral1}+\uppercase\expandafter{\romannumeral2}|\le CN^{-k}\Vert v^N\Vert_\varepsilon.
	\end{equation}
\end{lemma}
\begin{proof}
	Consider \eqref{decomposition of u} and \eqref{decomposition of new interpolation u}, $\Pi u-u$ can be split into
	\begin{equation*}
		\Pi u-u= (\pi_1E_1-E_1) +(\pi_2E_2-E_2)+(\pi_{12}E_{12}-E_{12})+(S^I-S).
	\end{equation*}
	
	Triangle inequality yields
	\begin{equation}
		\Vert \Pi u-u\Vert^2\le \Vert\pi_1E_1-E_1\Vert^2+\Vert\pi_2E_2-E_2\Vert^2+\Vert\pi_{12}E_{12}-E_{12}\Vert^2+\Vert S^I-S\Vert^2.
	\end{equation}
	
	Triangle inequality, Lemma \ref{error estimates}, Lemma \ref{step} and \eqref{new interpolation-1} yield
	\begin{equation}\label{E1-1}
		\begin{aligned}
			\Vert \pi_1E_1-E_1 \Vert^2&\le\Vert \pi_1E_1-E_1^I\Vert^2+\Vert E_1^I-E_1\Vert^2\le CN^{-2(k+1)},
		\end{aligned}
	\end{equation}
	where we have used 
	\begin{equation*}
		\begin{aligned}
			\Vert \pi_1E_1-E_1^I\Vert^2&\le CN^{-2\sigma}\sum_{s=0}^{k-1}\left(\sum_{t=1}^{k-1}\Vert\theta_{N/2-1,0}^{s,t}\Vert^2+\sum_{j=1}^{N-1}\sum_{t=0}^{k-1}\Vert\theta_{N/2-1,j}^{s,t}\Vert^2\right)\\
			&\le CN^{-2\sigma}\sum_{j=0}^{N-1}(h_{x,N/2-1}h_{y,j})\\
			&\le CN^{-2\sigma-1}.
		\end{aligned}
	\end{equation*}
	
	In a similar manner as $\Vert \pi_1E_1-E_1 \Vert^2$, we can obtain
	\begin{align}
		&\Vert \pi_2E_2-E_2 \Vert^2\le CN^{-2(k+1)},\label{E2-1}\\
		&\Vert \pi_{12}E_{12}-E_{12} \Vert^2\le C\varepsilon^2N^{-2k}+C\varepsilon N^{-2(k+1)}+ CN^{-4\sigma-2}\label{E12-1}.
	\end{align}
	
	\begin{equation}\label{S-1}
		\Vert S^I-S\Vert^2\le CN^{-2(k+1)},
	\end{equation}
	can be deduced by Lemma \ref{interpolation errors}, Lemma \ref{bound} and Assumption \ref{bound}.
	
	Collecting \eqref{E1-1}-\eqref{S-1}, we can prove
	\begin{equation}\label{u-1}
		\Vert\Pi u-u\Vert\le CN^{-(k+1)}.
	\end{equation}
	
	$\Vert \Pi u-u\Vert_\varepsilon$ can be analyzed similarly as above, then one has
	\begin{equation}\label{u-2}
		\Vert\Pi u-u\Vert_\varepsilon\le CN^{-k}.
	\end{equation}
	
	In conclusion,
	\begin{equation}
		|\uppercase\expandafter{\romannumeral1}+\uppercase\expandafter{\romannumeral2}|\le C\Vert\Pi u-u\Vert_\varepsilon\Vert v^N\Vert_\varepsilon\le CN^{-k}\Vert v^N\Vert_\varepsilon.
	\end{equation}
\end{proof}

In the next lemma, a novel analysis technique will be established to get an optimal result for \uppercase\expandafter{\romannumeral3} and \uppercase\expandafter{\romannumeral4}.
\begin{lemma}\label{main-2}
	Let Assumptions \ref{bound} and \ref{restriction} hold. Let $\sigma\ge k+1$. Then one can obtain
	\begin{equation}
		|\uppercase\expandafter{\romannumeral3}+\uppercase\expandafter{\romannumeral4}|\le CN^{-(k+1)}\Vert v^N\Vert_\varepsilon+C\varepsilon^{\sigma-1/2}N^{-1/2}\Vert v^N\Vert_\varepsilon.
	\end{equation}
\end{lemma}
\begin{proof}
	
	First, we will analyze $\int_{K_a}b_1(\Pi u-u)_xv^N\mathrm{d}x\mathrm{d}y$, which is the main focus of this article. Standard error estimates, i.e., studying $\int_{K_a}b_1(S^I-S)_xv^N\mathrm{d}x\mathrm{d}y$ and $\int_{K_a}b_1(\pi_iE_i-E_i)_xv^N\mathrm{d}x\mathrm{d}y\ (i=1,2,12)$ separately, would result in the inability of \begin{equation}\label{intractable}
		|\int_{K_a}b_1\sum_{s=0}^{k-1}E_1(x_{N/2-1}^s,y_N^0)(\theta_{N/2-1,N}^{s,0})_x v^N\mathrm{d}x\mathrm{d}y|\le Ch_{x,N/2-1}^{-1/2}h_{y,N-1}^{1/2}N^{-\sigma}\Vert v^N\Vert_\varepsilon
	\end{equation} to achieve uniform convergence, where $\frac{1}{2}\sigma\varepsilon\le h_{x,N/2-1}\le 2\sigma N^{-1}$ and $N^{-1}\le h_{y,N-1}\le 2N^{-1}$. 
	
	To avoid analyzing \eqref{intractable} and to get an optimal result, we propose a new analysis technique. Firstly, according to the boundary condition $\Pi u=0\ {\text on}\ \partial\Omega$ and \eqref{decomposition of new interpolation u}-\eqref{new interpolation-3}, we derive that $\Pi u|_{K_a}$ satisfies:
	\begin{equation*}
		\left\{
		\begin{aligned}
			&\Pi u(x_{N/2-1}^s,y_N^0)=0 \quad s=0,1,\dots,k-1,\\
			&\Pi u(x_{N/2}^0,y_N^0)=0,\\
			&\Pi u(x_{N/2}^0,y_{N-1}^t)=u(x_{N/2}^0,y_{N-1}^t)\quad t=0,1,\dots,k-1,\\
			&\Pi u(x_{N/2-1}^s,y_{N-1}^t)=u(x_{N/2-1}^s,y_{N-1}^t)-E_1(x_{N/2-1}^s,y_{N-1}^t) \quad s,t=0,1,\dots,k-1,
		\end{aligned}
		\right.
	\end{equation*}
	and its specific form is
	\begin{align*}
		\Pi u=&\sum_{t=0}^{k-1}u(x_{N/2}^0,y_{N-1}^t)\theta_{N/2,N-1}^{0,t}+\sum_{s=0}^{k-1}\sum_{t=0}^{k-1}u(x_{N/2-1}^s,y_{N-1}^t)\theta_{N/2-1,N-1}^{s,t}\\
		&-\sum_{s=0}^{k-1}\sum_{t=0}^{k-1}E_1(x_{N/2-1}^s,y_{N-1}^t)\theta_{N/2-1,N-1}^{s,t}.
	\end{align*}
	
	Then, decompose $\int_{K_a}b_1(\Pi u-u)_xv^N\mathrm{d}x\mathrm{d}y$ as the smooth part $\mathcal{B}_S$ and the layer part $\mathcal{B}_E$:
	\begin{equation*}
		\int_{K_a}b_1(\Pi u-u)_xv^N\mathrm{d}x\mathrm{d}y=\int_{K_a}b_1\left((\Pi u)_x-u_x\right)v^N\mathrm{d}x\mathrm{d}y=\mathcal{B}_S+\mathcal{B}_E,
	\end{equation*}
	where
	\begin{equation*}
			\begin{aligned}
				\mathcal{B}_S:=&\int_{K_a}b_1\left[\sum_{t=0}^{k-1}S(x_{N/2}^0,y_{N-1}^t)(\theta_{N/2,N-1}^{0,t})_x\right.\\
				&\left.+\sum_{s=0}^{k-1}\sum_{t=0}^{k-1}S(x_{N/2-1}^s,y_{N-1}^t)(\theta_{N/2-1,N-1}^{s,t})_x-S_x(x,y)\right] v^N\mathrm{d}x\mathrm{d}y,
			\end{aligned}
		\end{equation*}
	and
	\begin{equation*}
			\begin{aligned}
				\mathcal{B}_E:=&\int_{K_a}b_1\left[\sum_{t=0}^{k-1}E(x_{N/2}^0,y_{N-1}^t)(\theta_{N/2,N-1}^{0,t})_x+\sum_{s=0}^{k-1}\sum_{t=0}^{k-1}E_2(x_{N/2-1}^s,y_{N-1}^t)(\theta_{N/2-1,N-1}^{s,t})_x\right.\\
				&\left.+\sum_{s=0}^{k-1}\sum_{t=0}^{k-1}E_{12}(x_{N/2-1}^s,y_{N-1}^t)(\theta_{N/2-1,N-1}^{s,t})_x-E_x(x,y)\right]v^N\mathrm{d}x\mathrm{d}y.
			\end{aligned}
	\end{equation*}
	
	After conducting the aforementioned steps, it is evident that there is no further requirement to address the intractable term \eqref{intractable}. Then the primary difficulty turns into analyzing $\mathcal{B}_S$. To obtain the optimal convergence order of $\mathcal{B}_S$, we convert its corresponding estimates in the two-dimensional case into several types of one-dimensional interpolation error estimates. The specific conversion is as follows:
	
	The piecewise tensor-product interpolation basis function $\theta_{i,j}^{s,t}(x,y)$ can be expressed as the product of two univariate basis functions: $\theta_i^s(x)\times\theta_j^t(y)$, where $\theta_i^s(x)$ and $\theta_j^t(y)$ represent the basis function at point $(x_i^s,y_j^t)$ in the $x$- and $y$-directions, respectively. In light of this, we can obtain
	\begin{equation*}
		\begin{aligned}
&\sum_{t=0}^{k-1}S(x_{N/2}^0,y_{N-1}^t)(\theta_{N/2,N-1}^{0,t})_x+\sum_{s=0}^{k-1}\sum_{t=0}^{k-1}S(x_{N/2-1}^s,y_{N-1}^t)(\theta_{N/2-1,N-1}^{s,t})_x-S_x(x,y)\\
=&\sum_{t=0}^{k-1}\theta_{N-1}^t(y)\left(S(x_{N/2}^0,y_{N-1}^t)(\theta_{N/2}^0(x))_x+\sum_{s=0}^{k-1}S(x_{N/2-1}^s,y_{N-1}^t)(\theta_{N/2-1}^s(x))_x-S_x(x,y_{N-1}^t)\right)\\
&+\left(\sum_{t=0}^{k-1}\theta_{N-1}^t(y)S_x(x,y_{N-1}^t)+\theta_{N}^0(y)S_x(x,y_{N}^0)-S_x(x,y)\right)-\theta_{N}^0(y)S_x(x,y_{N}^0)\\
=&\sum_{t=0}^{k-1}\theta_{N-1}^t(y)\left(S_x^I(x,y_{N-1}^t)-S_x(x,y_{N-1}^t)\right)+\left(S_x^I(x,y)-S_x(x,y)\right)-\theta_{N}^0(y)S_x(x,y_{N}^0),
    \end{aligned}
\end{equation*}
here $S_x^I(x,y_{N-1}^t)-S_x(x,y_{N-1}^t)$ denotes the one-dimensional interpolation of variable $x$, and $S_x^I(x,y)-S_x(x,y)$ denotes the one-dimensional interpolation of variable $y$. 
	
The follwing one-dimensional interpolation errors (\cite[Theorem 3.1.4]{Cia:2002-Finite}) will be employed for the next estimates:
	\begin{align}
		&\Vert v-v^I\Vert_{W^{l,q}}(I_i)\le h_{x,i}^{k+1-l+1/q-1/p}|v|_{W^{k+1,p}(I_i)}\quad\forall v\in W^{k+1,p}(I_i),\label{x-direction error estimate}\\
		&\Vert v-v^I\Vert_{W^{l,q}}(J_j)\le h_{y,j}^{k+1-l+1/q-1/p}|v|_{W^{k+1,p}(J_j)}\quad\forall v\in W^{k+1,p}(J_j)\label{y-direction error estimate},
	\end{align}
	where $I_i$ denotes $[x_i,x_{i+1}]$ with $i=0,1,\dots,N-1$ and $J_j$ denotes $[y_j,y_{j+1}]$ with $j=0,1,\dots,N-1$, $l=0,1$ and $1\le p,q\le\infty$. 
	
	Now, we present a detailed demonstration for $\int_{K_a}b_1(\Pi u-u)_xv^N\mathrm{d}x\mathrm{d}y$, which can be decomposed as
	\begin{equation}\label{initial}
		\begin{aligned}
			&\int_{K_a}b_1(\Pi u-u)_xv^N\mathrm{d}x\mathrm{d}y\\
			=:&\mathcal{I}_1+\mathcal{I}_2+\mathcal{I}_3+\mathcal{I}_4+\mathcal{I}_5+\mathcal{I}_6+\mathcal{I}_7,
		\end{aligned}
	\end{equation}
	where
	\begin{align}
		&\mathcal{I}_1=\int_{K_a}b_1\sum_{t=0}^{k-1}\theta_{N-1}^t(y)\left(S_x^I(x,y_{N-1}^t)-S_x(x,y_{N-1}^t)\right)v^N\mathrm{d}x\mathrm{d}y,\label{remark-2}\\
		&\mathcal{I}_2=\int_{K_a}b_1\left( S_x^I(x,y)-S_x(x,y)\right)v^N\mathrm{d}x\mathrm{d}y,\label{remark-22}\\
		&\mathcal{I}_3=-\int_{K_a}b_1\theta_{N}^0(y)S_x(x,y_N^0)v^N\mathrm{d}x\mathrm{d}y,\label{remark-3}\\
		&\mathcal{I}_4=\int_{K_a}b_1\sum_{t=0}^{k-1}E(x_{N/2}^0,y_{N-1}^t)(\theta_{N/2,N-1}^{0,t})_xv^N\mathrm{d}x\mathrm{d}y,\label{remark-4}\\
		&\mathcal{I}_5=\int_{K_a}b_1\sum_{s=0}^{k-1}\sum_{t=0}^{k-1}E_2(x_{N/2-1}^s,y_{N-1}^t)(\theta_{N/2-1,N-1}^{s,t})_xv^N\mathrm{d}x\mathrm{d}y,\\
		&\mathcal{I}_6=\int_{K_a}b_1\sum_{s=0}^{k-1}\sum_{t=0}^{k-1}E_{12}(x_{N/2-1}^s,y_{N-1}^t)(\theta_{N/2-1,N-1}^{s,t})_xv^N\mathrm{d}x\mathrm{d}y,\\
		&\mathcal{I}_7=-\int_{K_a}b_1E_x(x,y)v^N\mathrm{d}x\mathrm{d}y\label{remark-5}.
	\end{align}
	
	Estimates for $\mathcal{I}_1-\mathcal{I}_7$ are discussed in detail below.
	
	From H\"{o}lder inequalities, \eqref{x-direction error estimate}, Assumption \ref{bound} and Lemma \ref{step}, one has
	\begin{equation}\label{begin-1}
		|\mathcal{I}_1|\le CN^{-k}\sum_{t=0}^{k-1}\Vert\frac{\mathrm{d}^{k+1}S(x,y_{N-1}^t)}{\mathrm{d}x^{k+1}}\Vert_{K_a}\Vert v^N\Vert_{K_a}
		\le CN^{-(k+1)}\Vert v^N\Vert_\varepsilon.
	\end{equation}
	
	H\"{o}lder inequalities, \eqref{y-direction error estimate}, Assumption \ref{bound} and Lemma \ref{step} give
	\begin{equation}\label{remark-6}
		|\mathcal{I}_2|
		\le CN^{-(k+1)}\Vert\frac{\mathrm{d}^{k+1}S_x(x,y)}{\mathrm{d}y^{k+1}}\Vert_{K_a}\Vert v^N\Vert_{K_a}
		\le CN^{-(k+2)}\Vert v^N\Vert_\varepsilon.
	\end{equation}
	
	Green's formula yields
	\begin{align*}
		\mathcal{I}_3=&\int_{K_a}((b_1)_xv^N+b_1v_x^N)\theta_{N}^0(y)S(x,y_N^0)\mathrm{d}x\mathrm{d}y-\int_{y_{N-1}}^{y_N}b_1\theta_{N}^0(y)S(x_{N/2},y_N^0)v^N(x_{N/2},y)\mathrm{d}y\\
		& +\int_{y_{N-1}}^{y_N}b_1\theta_{N}^0(y)S(x_{N/2-1},y_N^0)v^N(x_{N/2-1},y)\mathrm{d}y.
	\end{align*}
	Since $u=0$ on the boundary $\partial\Omega$, the smooth function $S(x,y_N^0)$ can be seen as the layer function $-E(x,y_N^0)$, which holds
	\begin{equation*}
		\Vert S(x,y_N^0)\Vert_{K_a}=\Vert E(x,y_N^0)\Vert_{K_a}\le C\varepsilon^{1/2}N^{-\sigma}.
	\end{equation*}
	Thus,
	\begin{equation}\label{remark-7}
		|\int_{K_a}((b_1)_xv^N+b_1v_x^N)\theta_{N}^0(y)S(x,y_N^0)\mathrm{d}x\mathrm{d}y|\le C\Vert S(x,y_N^0)\Vert_{K_a}\Vert v_x\Vert_{K_a}\le CN^{-\sigma}\Vert v^N\Vert_\varepsilon.
	\end{equation}
	Besides,
	\begin{equation}\label{remark-8}
		\begin{aligned}
			&|\int_{y_{N-1}}^{y_N}b_1\theta_{N}^0(y)S(x_{N/2},y_N^0)v^N(x_{N/2},y)\mathrm{d}y|+|\int_{y_{N-1}}^{y_N}b_1\theta_{N}^0(y)S(x_{N/2-1},y_N^0)v^N(x_{N/2-1},y)\mathrm{d}y|\\
			\le& C\varepsilon^{\sigma}\Vert v^N\Vert_\varepsilon+CN^{-\sigma-1/2}\ln^{1/2}N\Vert v^N\Vert_\varepsilon,
		\end{aligned}
	\end{equation}
	where we have used $|S(x_{N/2},y_N^0)|=|-E(x_{N/2},y_N^0)|\le C\varepsilon^\sigma$, $|S(x_{N/2-1},y_N^0)|=|-E(x_{N/2-1},y_N^0)|\le CN^{-\sigma}$ and 
	\begin{equation}\label{partrial method}
		\begin{aligned}
			&\begin{aligned}
				|\int_{y_{N-1}}^{y_N}v^N(x_{N/2},y)\mathrm{d}y|\le CN\Vert v^N\Vert_{L^1(K_{N/2,N-1})}\le C\Vert v^N\Vert_\varepsilon,
			\end{aligned}\\
			&\begin{aligned}
				|\int_{y_{N-1}}^{y_N}v^N(x_{N/2-1},y)\mathrm{d}y|&\le C|\int_{y_{N-1}}^{y_N}\int_{x_0}^{x_{N/2-1}}v_x^N(x,y)\mathrm{d}x\mathrm{d}y|\\
				&\le C \Vert v_x^N\Vert_{L^1([x_0,x_{N/2-1}]\times[y_{N-1},y_N])}\\
				&\le CN^{-1/2}\ln^{1/2}N\Vert v^N\Vert_\varepsilon.
			\end{aligned}
		\end{aligned}
	\end{equation} 
\eqref{remark-7}-\eqref{remark-8} give 
	\begin{equation}\label{remark-9}
		|\mathcal{I}_3|\le CN^{-\sigma}\Vert v^N\Vert_\varepsilon.
	\end{equation}
	
	By $|E(x_{N/2},y_{N-1}^t)|\le C\varepsilon^{\sigma}$, $|E_2(x_{N/2-1}^s,y_{N-1}^t)|\le C\varepsilon^\sigma$, $|E_{12}(x_{N/2-1}^s,y_{N-1}^t)|\le C\varepsilon^\sigma N^{-\sigma}$, inverse inequalities and Lemma \ref{step}, we have
	\begin{align}
		&|\mathcal{I}_4|\le C\varepsilon^{\sigma}(h_{x,N/2-1}^{-1}h_{x,N/2-1}^{1/2}h_{y,N-1}^{1/2})\Vert v^N\Vert_{K_a}\le C\varepsilon^{\sigma-1/2}N^{-1/2}\Vert v^N\Vert_\varepsilon,\label{remark-1}\\
		&|\mathcal{I}_5|\le C\varepsilon^\sigma (h_{x,N/2-1}^{-1}h_{x,N/2-1}^{1/2}h_{y,N-1}^{1/2})\Vert v^N\Vert_{K_a}\le C\varepsilon^{\sigma-1/2}N^{-1/2}\Vert v^N\Vert_\varepsilon,\\
		&|\mathcal{I}_6|\le C\varepsilon^\sigma N^{-\sigma} (h_{x,N/2-1}^{-1}h_{x,N/2-1}^{1/2}h_{y,N-1}^{1/2})\Vert v^N\Vert_{K_a}\le
		C\varepsilon^{\sigma-1/2}N^{-\sigma-1/2}\Vert v^N\Vert_\varepsilon\label{remark-11}.
	\end{align}
	
	Green's formula yields
	\begin{align*}
		\mathcal{I}_7=&\int_{K_a}((b_1)_xv^N+b_1v_x^N)E(x,y)\mathrm{d}x\mathrm{d}y-\int_{y_{N-1}}^{y_N}b_1E(x_{N/2},y)v^N(x_{N/2},y)\mathrm{d}y\\
		&+\int_{y_{N-1}}^{y_N}b_1E(x_{N/2-1},y)v^N(x_{N/2-1},y)\mathrm{d}y,
	\end{align*}
	where
	\begin{equation}\label{remark-10}
		|\int_{K_a}((b_1)_xv^N+b_1v_x^N)E(x,y)\mathrm{d}x\mathrm{d}y|\le CN^{-\sigma}\Vert v^N\Vert_\varepsilon,
	\end{equation}
	and 
	\begin{equation}\label{finish-1}
		\begin{aligned}
			&|\int_{y_{N-1}}^{y_N}b_1E(x_{N/2},y)v^N(x_{N/2},y)\mathrm{d}y|+|\int_{y_{N-1}}^{y_N}b_1E(x_{N/2-1},y)v^N(x_{N/2-1},y)\mathrm{d}y|\\
			\le&C\varepsilon^\sigma|\int_{y_{N-1}}^{y_N}v^N(x_{N/2},y)\mathrm{d}y|+CN^{-\sigma}|\int_{y_{N-1}}^{y_N}v^N(x_{N/2-1},y)\mathrm{d}y|\\
			\le& C\varepsilon^{\sigma}\Vert v^N\Vert_\varepsilon+CN^{-\sigma-1/2}\ln^{1/2}N\Vert v^N\Vert_\varepsilon.
		\end{aligned}
	\end{equation}
	\eqref{remark-10}-\eqref{finish-1} give
	\begin{equation}\label{remark-12}
		|\mathcal{I}_7|\le CN^{-\sigma}\Vert v^N\Vert_\varepsilon.
	\end{equation}
	
	Substituting \eqref{begin-1}-\eqref{remark-6}, \eqref{remark-9}, \eqref{remark-1}-\eqref{remark-11} and \eqref{remark-12} into \eqref{initial}, we prove
	\begin{equation}\label{aa-2}
		|\int_{K_a}b_1(\Pi u-u)_xv^N\mathrm{d}x\mathrm{d}y|\le CN^{-(k+1)}\Vert v^N\Vert_\varepsilon+C\varepsilon^{\sigma-1/2}N^{-1/2}\Vert v^N\Vert_\varepsilon.
	\end{equation}
	
	Next, we are going to deal with $\int_{K_a}b_2(\Pi u-u)_yv^N\mathrm{d}x\mathrm{d}y$, which can be divided into 
	\begin{equation}\label{aa-1}
		\begin{aligned}
			&\int_{K_a}b_2(\Pi u-u)_yv^N\mathrm{d}x\mathrm{d}y\\
			=&\int_{K_a}b_2(S^I-S)_yv^N\mathrm{d}x\mathrm{d}y+\int_{K_a}b_2(\pi_1E_1-E_1)_yv^N\mathrm{d}x\mathrm{d}y+\int_{K_a}b_2(E_2^I-E_2)_yv^N\mathrm{d}x\mathrm{d}y\\
			&+\int_{K_a}b_2(E_{12}^I-E_{12})_yv^N\mathrm{d}x\mathrm{d}y.
		\end{aligned}
	\end{equation}
	
	\begin{equation}\label{begin-2}
		|\int_{K_a}b_2(S^I-S)_yv^N\mathrm{d}x\mathrm{d}y|\le CN^{-(k+1)}\Vert v^N\Vert_\varepsilon,
	\end{equation}
	is straightforward by H\"{o}lder inequalities, Lemma \ref{interpolation errors}, Assumption \ref{bound} and Lemma \ref{step}. 
	
	From H\"{o}lder inequalities, inverse inequalities, Lemma \ref{step} and \eqref{max-2}, we have
	\begin{equation}
		|\int_{K_a}b_2(\pi_1E_1-E_1)_yv^N\mathrm{d}x\mathrm{d}y|\le CN\Vert E_1\Vert_{\infty,K_a}\Vert v^N\Vert_{L^1(K_a)}\le CN^{-\sigma}\Vert v^N\Vert_\varepsilon.
	\end{equation}
	
	Under a similar analysis as $\int_{K_a}b_2(\pi_1E_1-E_1)_yv^N\mathrm{d}x\mathrm{d}y$, we can get
	\begin{align}
		&|\int_{K_a}b_2(E_2^I-E_2)_yv^N\mathrm{d}x\mathrm{d}y|\le C\varepsilon^{\sigma}\Vert v^N\Vert_\varepsilon,\\
		&|\int_{K_a}b_2(E_{12}^I-E_{12})_yv^N\mathrm{d}x\mathrm{d}y|\le C\varepsilon^{\sigma}N^{-\sigma}\Vert v^N\Vert_\varepsilon.\label{finish-2}
	\end{align}
	
	Substituting \eqref{begin-2}-\eqref{finish-2} into \eqref{aa-1}, we prove
	\begin{equation}\label{finish-3}
		|\int_{K_a}b_2(\Pi u-u)_yv^N\mathrm{d}x\mathrm{d}y|\le CN^{-(k+1)}\Vert v^N\Vert_\varepsilon.
	\end{equation}

	In conclusion, combined \eqref{aa-2} with \eqref{finish-3}, we obtain
	\begin{equation}
		|\uppercase\expandafter{\romannumeral3}|=|-\int_{K_a}\boldsymbol{b}\cdot\nabla (\Pi u-u)v^N\mathrm{d}x\mathrm{d}y|\le CN^{-(k+1)}\Vert v^N\Vert_\varepsilon+C\varepsilon^{\sigma-1/2}N^{-1/2}\Vert v^N\Vert_\varepsilon.
	\end{equation}
	
	Derivations of \uppercase\expandafter{\romannumeral4} is similar to that of \uppercase\expandafter{\romannumeral3}, thus we have
	\begin{equation}
		|\uppercase\expandafter{\romannumeral4}|=|-\int_{K_b}\boldsymbol{b}\cdot\nabla (\Pi u-u)v^N\mathrm{d}x\mathrm{d}y|\le CN^{-(k+1)}\Vert v^N\Vert_\varepsilon+C\varepsilon^{\sigma-1/2}N^{-1/2}\Vert v^N\Vert_\varepsilon.
	\end{equation}
\end{proof}
\begin{remark}
	We briefly summarize the novel analysis technique. Firstly, the specific form of the interpolation $\Pi u$ on $K_a$ is described. Then, decompose $\int_{K_a}b_1(\Pi u-u)_xv^N\mathrm{d}x\mathrm{d}y$ as the smooth part $\mathcal{B}_S$ and the layer part $\mathcal{B}_E$. For $\mathcal{B}_S$, we convert the corresponding estimates in the two-dimensional case into several types of one-dimensional interpolation error estimates, see \eqref{remark-2}-\eqref{remark-22} and \eqref{begin-1}-\eqref{remark-6}. Furthermore, we take full advantage of the boundary condition where the sum of the smooth function and the layer function is zero, so that analysis of the smooth function on the boundary can be transformed into studying the layer function, see \eqref{remark-3} and \eqref{remark-7}-\eqref{remark-8}.
	For $\mathcal{B}_E$, we employ the standard error estimats, see \eqref{remark-4}-\eqref{remark-5} and \eqref{remark-1}-\eqref{remark-12}.
\end{remark}

\begin{lemma}\label{main-3}
	Let Assumptions \ref{bound} and \ref{restriction} hold. Let $\sigma\ge k+1$. Then we have
	\begin{equation*}
		|\uppercase\expandafter{\romannumeral5}+\uppercase\expandafter{\romannumeral6}|\le CN^{-k}\Vert v^N\Vert_\varepsilon.
	\end{equation*}
\end{lemma}
\begin{proof}
	To begin with, we discuss term $\int_{\Omega_0}\boldsymbol{b}\cdot\nabla(\pi_1E_1-E_1)v^N\mathrm{d}x\mathrm{d}y$. Green's formula generates
	\begin{equation}\label{bb-1}
		\begin{aligned}
			&\int_{\Omega_0}\boldsymbol{b}\cdot\nabla(\pi_1E_1-E_1)v^N\mathrm{d}x\mathrm{d}y\\
			=&-\int_{\Omega_0}(\nabla\cdot\boldsymbol{b})(\pi_1E_1-E_1)v^N\mathrm{d}x\mathrm{d}y-\int_{\Omega_0}(\pi_1E_1-E_1)(\boldsymbol{b}\cdot\nabla v^N)\mathrm{d}x\mathrm{d}y\\
			&+\int_{\partial\Omega_0}(\boldsymbol{b}\cdot \boldsymbol{n})(\pi_1E_1-E_1)v^N\mathrm{d}s.
		\end{aligned}
	\end{equation}
	
	Employing H\"{o}lder inequalities and \eqref{E1-1}, we have
	\begin{equation}\label{qq-1}
		|-\int_{\Omega_0}(\nabla\cdot\boldsymbol{b})(\pi_1E_1-E_1)v^N\mathrm{d}x\mathrm{d}y|\le CN^{-(k+1)}\Vert v^N\Vert_\varepsilon.
	\end{equation}
	
	The reader is referred to \cite[Lemma 4.1]{zhang1Liu2:2023-Convergence} for detailed derivations of $\int_{\Omega_0}(\pi_1E_1-E_1)(\boldsymbol{b}\cdot\nabla v^N)\mathrm{d}x\mathrm{d}y$, from which one can get
	\begin{equation}\label{bb-4}
		|\int_{\Omega_0}(\pi_1E_1-E_1)(\boldsymbol{b}\cdot\nabla v^N)\mathrm{d}x\mathrm{d}y|\le CN^{-(k+1/2)}\Vert v^N\Vert_\varepsilon.
	\end{equation}
	
	Then take $\int_{\partial\Omega_0}(\boldsymbol{b}\cdot \boldsymbol{n})(\pi_1E_1-E_1)v^N\mathrm{d}s$ into consideration:
	
	\begin{equation}\label{qq-5}
		\begin{aligned}
			&|\int_{\partial\Omega_0}(\boldsymbol{b}\cdot \boldsymbol{n})(\pi_1E_1-E_1)v^N\mathrm{d}s|\\
			\le&C\Vert \pi_1E_1-E_1\Vert_{L^\infty(\partial\Omega_0\textbackslash\partial\Omega)}\Vert v^N\Vert_{L^1(\partial\Omega_0\textbackslash\partial\Omega)}\\
			\le&C\Vert E_1\Vert_{L^\infty(\partial\Omega_0\textbackslash\partial\Omega)}\Vert v^N\Vert_{L^1(\partial\Omega_0\textbackslash\partial\Omega)}\\
			\le&CN^{-\sigma}\left(|\int_{y_{N/2-1}}^{y_{N/2}}v^N(x_{N-1},y)\mathrm{d}y|+|\int_{y_{N-1}}^1v^N(x_{N/2},y)\mathrm{d}y|+|\int_{y_{N-1}}^1v^N(x_{N/2-1},y)\mathrm{d}y|\right)\\
			&+CN^{-\sigma}\left(|\int_{x_{N/2-1}}^{x_{N/2}}v^N(x,y_{N-1})\mathrm{d}x|+|\int_{x_{N-1}}^1v^N(x,y_{N/2})\mathrm{d}x|+|\int_{x_{N-1}}^1v^N(x,y_{N/2-1})\mathrm{d}x|\right)\\
			\le&CN^{-\sigma}\Vert v^N\Vert_\varepsilon,
		\end{aligned}
	\end{equation}
	here we have used H\"{o}ler inequalities, \eqref{max-2}, Lemma \ref{step} and \eqref{partrial method}.
	
	From \eqref{qq-1}-\eqref{qq-5}, one can get
	\begin{equation}\label{qq-6}
		|\int_{\Omega_0}\boldsymbol{b}\cdot\nabla(\pi_1E_1-E_1)v^N\mathrm{d}x\mathrm{d}y|\le CN^{-(k+1/2)}\Vert v^N\Vert_\varepsilon.
	\end{equation}
	
	Similarly, 
	\begin{equation}\label{qq-7}
		|\int_{\Omega_0}\boldsymbol{b}\cdot\nabla(\pi_2E_2-E_2)v^N\mathrm{d}x\mathrm{d}y|\le CN^{-(k+1/2)}\Vert v^N\Vert_\varepsilon.
	\end{equation}
	
	Next, we study $\int_{\Omega_0}\boldsymbol{b}\cdot\nabla(\pi_{12}E_{12}-E_{12})v^N\mathrm{d}x\mathrm{d}y$, which can be divided into the follwing terms using Green's formula:
	\begin{equation}\label{yy-1}
		\begin{aligned}
			&\int_{\Omega_0}\boldsymbol{b}\cdot\nabla(\pi_{12}E_{12}-E_{12})v^N\mathrm{d}x\mathrm{d}y\\
			=&-\int_{\Omega_0}(\nabla\cdot\boldsymbol{b})(\pi_{12}E_{12}-E_{12})v^N\mathrm{d}x\mathrm{d}y-\int_{\Omega_0}(\pi_{12}E_{12}-E_{12})(\boldsymbol{b}\cdot\nabla v^N)\mathrm{d}x\mathrm{d}y\\
			&+\int_{\partial\Omega_0}(\boldsymbol{b}\cdot \boldsymbol{n})(\pi_{12}E_{12}-E_{12})v^N\mathrm{d}s.
		\end{aligned}
	\end{equation}
	
	According to H\"{o}lder inequalities and \eqref{E12-1}, one obtains
	\begin{equation}\label{yy-2}
		|-\int_{\Omega_0}(\nabla\cdot\boldsymbol{b})(\pi_{12}E_{12}-E_{12})v^N\mathrm{d}x\mathrm{d}y|\le (C\varepsilon N^{-k}+C\varepsilon^{1/2}N^{-(k+1)}+CN^{-(2\sigma+1)})\Vert v^N\Vert_\varepsilon.
	\end{equation}
	
	Refer to \cite[Lemma 4.2]{zhang1Liu2:2023-Convergence} for detailed demonstrations of the second term in the right-hand side of \eqref{yy-1}:
	\begin{equation}
		|\int_{\Omega_0}(\pi_{12}E_{12}-E_{12})(\boldsymbol{b}\cdot\nabla v^N)\mathrm{d}x\mathrm{d}y|\le CN^{-(k+1/2)}\Vert v^N\Vert_\varepsilon.
	\end{equation}
	
	With an approach similar to \eqref{qq-5}, we can derive that
	\begin{equation}\label{yy-3}
		|\int_{\partial\Omega_0}(\boldsymbol{b}\cdot \boldsymbol{n})(\pi_{12}E_{12}-E_{12})v^N\mathrm{d}s|\le CN^{-\sigma}\Vert v^N\Vert_\varepsilon.
	\end{equation}
	
	Substituting \eqref{yy-2}-\eqref{yy-3} into \eqref{yy-1}, one can obtain
	\begin{equation}\label{yy-4}
		|\int_{\Omega_0}\boldsymbol{b}\cdot\nabla(\pi_{12}E_{12}-E_{12})v^N\mathrm{d}x\mathrm{d}y|\le CN^{-(k+1/2)}\Vert v^N\Vert_\varepsilon.
	\end{equation}
	
	\eqref{yy-4}, combined with \eqref{qq-6} and \eqref{qq-7}, proves
	\begin{equation}
		|\uppercase\expandafter{\romannumeral5}|=|-\sum_{i=1,2,12}\int_{\Omega_0}\boldsymbol{b}\cdot\nabla(\pi_iE_i-E_i) v^N\mathrm{d}x\mathrm{d}y|\le CN^{-(k+1/2)}\Vert v^N\Vert_\varepsilon.
	\end{equation}
	
	At last, H\"{o}lder inequalities, Lemma \ref{interpolation errors}, Lemma \ref{step} and Assumption \ref{bound} yield
	\begin{equation}
		|\uppercase\expandafter{\romannumeral6}|=|-\int_{\Omega_0}\boldsymbol{b}\cdot\nabla(S^I-S) v^N\mathrm{d}x\mathrm{d}y|\le C\Vert\nabla(S^I-S)\Vert_{\Omega_0}\Vert v^N\Vert_{\Omega_0}\le CN^{-k}\Vert v^N\Vert_\varepsilon.
	\end{equation}
	
	Thus we are done.
\end{proof}

Now we are in a position to convey the main conclusion of this article.
\begin{theorem}\label{main theorem}
	Let Assumptions \ref{bound} and \ref{restriction} hold. Let $u$ denote the exact solution to \eqref{model problem} and $u^N$ denote the corresponding finite element solution to \eqref{variational form 2}. Then we have
	\begin{equation*}
		\Vert u-u^N\Vert_\varepsilon\le CN^{-k}.
	\end{equation*}
\end{theorem}
\begin{proof}
	Triangle inequality generates
	\begin{equation*}
		\Vert u-u^N\Vert_\varepsilon\le\Vert u-\Pi u\Vert_\varepsilon+\Vert \Pi u-u^N\Vert_\varepsilon,
	\end{equation*}
	where from \eqref{u-2}, one has
	\begin{equation}
		\Vert u-\Pi u\Vert_\varepsilon\le CN^{-k},
	\end{equation}
	and from Lemmas \ref{main-1}, \ref{main-2} and \ref{main-3}, one has
	\begin{equation}
		\Vert \Pi u-u^N\Vert_\varepsilon\le CN^{-k}.
	\end{equation}
	
	Thus we are done.
\end{proof}
\section{Numerical experiments}\label{sec 5}
Numerical experiments are carried out in this section to confirm our theoretical results. All calculations were performed using Intel Visual Fortran 11, and discrete problems were solved with the aid of the nonsymmetric iterative solver GMRES; see, e.g., \cite{Ben1Gol2:2005-Numerical}.

Consider the following singularly perturbed convection-diffusion equation:
\begin{equation}\label{model problem eg}
	\begin{aligned}
		-\varepsilon\Delta u-(2+x-y)u_x-(2-x+y)u_y+2u=&f\quad&&\text{in}\quad\Omega=(0,1)^2,\\
		u=&0\quad&&\text{on}\quad\partial\Omega,
	\end{aligned}
\end{equation}
where we select an appropriate function $f$ such that
\begin{equation}
	u(x,y)=2\sin(\pi x)(1-e^{-\frac{2x}{\varepsilon}})(1-y)^2(1-e^{-\frac{y}{\varepsilon}})
\end{equation}
is the exact solution to \eqref{model problem}. This solution typically exhibits exponential layers, as stated in Assumption \ref{bound}. 

Errors for $k=1$ are listed in Table \ref{table:1} under the energy norm $\Vert u-u^N\Vert_\varepsilon$, for $\varepsilon=10^{-4}, 10^{-5},\dots,10^{-8}$ and $N=8, 16,32,64,128,256$; Errors for $k=2$ are listed in Table \ref{table:2} under the energy norm $\Vert u-u^N\Vert_\varepsilon$, for $\varepsilon=10^{-4}, 10^{-5},\dots,10^{-8}$ and $N=8,16,32,64,128$. These data indicate that the finite element solution converges uniformly to the exact solution at an optimal rate of order $k+1$ under the energy norm, confirming our main conclusion, i.e., Theorem \ref{main theorem}.

\begin{table}[H]
	\caption{Errors of $\Vert u-u^N\Vert_{\varepsilon}$ and convergence order for $k=1$}
	\footnotesize
	\begin{tabular*}{\textwidth}{@{}@{\extracolsep{\fill}} c cccccc @{}}
		\cline{1-7}{}
		{ $\varepsilon$ }&\multicolumn{6}{c}{$N$ }\\ 
		\cline{2-7}                 &8        &16         &32        &64       &128        &256           \\
		\cline{1-7}
		{ $10^{-4}$ }&0.339E+00  &0.167E+00    &0.834E-01   &0.418E-01 &0.208E-01   &0.104E-01 \\
		&1.02    &1.00      &1.00      &1.00    &1.00   &---       \\
		\cline{2-7}
		{ $10^{-5}$ }&0.339E+00  &0.167E+00    &0.834E-01   &0.418E-01 &0.208E-01   &0.104E-01 \\
		&1.02    &1.00      &1.00      &1.00    &1.00   &---       \\
		\cline{2-7}
		{ $10^{-6}$ }&0.339E+00  &0.167E+00    &0.834E-01   &0.418E-01 &0.208E-01   &0.104E-01 \\
		&1.02    &1.00      &1.00      &1.00    &1.00   &---       \\
		\cline{2-7}
		{ $10^{-7}$}&0.339E+00  &0.167E+00    &0.834E-01   &0.418E-01 &0.208E-01   &0.104E-01 \\
		&1.02    &1.00      &1.00      &1.00    &1.00   &---       \\
		\cline{2-7}
		{ $10^{-8}$}&0.339E+00  &0.167E+00    &0.834E-01   &0.418E-01 &0.208E-01   &0.104E-01 \\
		&1.02    &1.00      &1.00      &1.00    &1.00   &---       \\
		\cline{1-7}{}
	\end{tabular*}
	\label{table:1}
\end{table}

\begin{table}[H]
	\caption{Errors of $\Vert u-u^N\Vert_{\varepsilon}$ and convergence order for $k=2$}
	\footnotesize
	\begin{tabular*}{\textwidth}{@{}@{\extracolsep{\fill}} c ccccc @{}}
		\cline{1-6}{}
		{ $\varepsilon$ }&\multicolumn{5}{c}{$N$ }\\ 
		\cline{2-6}                 &8        &16         &32        &64       &128                \\
		\cline{1-6}
		{ $10^{-4}$ }&0.103E+00  &0.257E-01    &0.643E-02   &0.162E-02 &0.470E-03    \\
		&2.00   &2.00      &1.99     &1.79      &---       \\
		\cline{2-6}
		{ $10^{-5}$ }&0.103E+00  &0.257E-01    &0.643E-02   &0.161E-02 &0.402E-03   \\
		&2.00   &2.00      &2.00     &2.00       &---       \\
		\cline{2-6}
		{ $10^{-6}$ }&0.103E+00  &0.257E-01    &0.643E-02   &0.161E-02 &0.402E-03   \\
		&2.00   &2.00      &2.00     &2.00       &---      \\
		\cline{2-6}
		{ $10^{-7}$}&0.103E+00  &0.257E-01    &0.643E-02   &0.161E-02 &0.402E-03  \\
		&2.00   &2.00      &2.00     &2.00       &---      \\
		\cline{2-6}
		{ $10^{-8}$}&0.103E+00  &0.257E-01    &0.643E-02   &0.161E-02 &0.402E-03  \\
		&2.00   &2.00      &2.00     &2.00       &---     \\
		\cline{1-6}{}
	\end{tabular*}
	\label{table:2}
\end{table}
\section{Conflict of interest statement}
We declare that we have no conflict of interest.
 
%
%
%

\bibliographystyle{spmpsci}      


\end{document}